\documentclass[12pt]{amsart}
\usepackage[]{amsmath, amsthm, amsfonts, verbatim, amssymb, tikz,inputenc}

\usepackage{amsmath}
\usepackage{amsthm}
\usepackage{amssymb}
\usepackage{amsfonts}
\usepackage{hyperref}
\usepackage{rotating}
\usepackage{amssymb}
\usepackage{epstopdf}
\usepackage{pifont}
\usepackage{amsmath}
\usepackage{enumerate}
\usepackage{graphicx}
\usepackage{color}
\usepackage{tikz-cd}
\usepackage{mathtools}

\DeclareMathAlphabet{\mathpzc}{OT1}{pzc}{m}{it}

\usepackage{tikz}
\usetikzlibrary{arrows}

\newtheorem{theorem}{Theorem}[section]
\newtheorem{lemma}[theorem]{Lemma}
\newtheorem{proposition}[theorem]{Proposition}
\newtheorem{corollary}[theorem]{Corollary}

\newtheorem{definition}[theorem]{Definition}

\newtheorem{question}[theorem]{Question}

\theoremstyle{remark}
\newtheorem{remark}[theorem]{Remark}
\newtheorem{example}[theorem]{Example}

\newcommand{\bN}{\mathbb{N}}
\newcommand{\bQ}{\mathbb{Q}}
\newcommand{\bP}{\mathbb{P}}

\newcommand{\bZ}{\mathbb{Z}}

\newcommand{\cF}{\mathcal{F}}
\newcommand{\cG}{\mathcal{G}}

\newcommand{\cO}{\mathcal{O}}

\newcommand{\Exc}{{\rm Exc}}

\newcommand{\Img}{{\rm Im}}

\newcommand{\Supp}{{\rm Supp}}

\def\<{\langle}
\def\>{\rangle}

\title{Positivity of anticanonical divisors in algebraic fibre spaces}
\author{Chi-Kang Chang}
\address{Department of Mathematics, National Taiwan University, Taipei, 106, Taiwan}
\email{d08221001@ntu.edu.tw}

\begin{document}
\maketitle
\begin{abstract}
Let $f:X\rightarrow Y$ be an algebraic fibre space between normal projective varieties and $F$ be a general fibre of $f$. We prove an Iitaka-type inequality $\kappa(X,-K_X)\leq \kappa(F,-K_F)+\kappa(Y,-K_Y)$ under some mild conditions. We also obtain results relating to the positivity of $-K_X$ and $-K_Y$.
\end{abstract}

\section{Introduction}
Given a surjective projective morphism $f:X\rightarrow Y$  between normal projective varieties, a natural and important problem is to compare the canonical divisors of $X$ and $Y$. In particular, the well-known Iitaka conjecture asserts that $$ \kappa(X) \ge \kappa(F)+\kappa(Y),$$
where $F$ denotes a general fibre. The main purpose of the conjecture is to relate the "positivities" of $K_X, K_Y$ and $K_F$. Additionally, it is also very natural and important to compare the positivities of $-K_X$ and $-K_Y$ especially in the case of Fano varieties and related varieties.

To study the above problems, recall that we have the weakly positivity theorem developed and generalized by \cite{C,F78,F17,K81,V}. These theorems show that for sufficiently divisible positive integer $m$, the sheaf $f_*\cO_X(m(K_{X/Y}))$ is a weakly positive sheaf in the sense of \cite{EG} (or \cite{V}), where $K_{X/Y}=K_X-f^*K_Y$. Thus, we should expect that the positivity of $K_Y$ will affect the positivity of $K_X$. Consequently, we should expect the positivity of $-K_X$ will affect the positivity of $-K_Y$.

The main purpose of this paper is to study the aspects of positivity of the anti-canonical divisors in algebraic fibre spaces. The main theorem below can be thought of as an analog of the Iitaka conjecture of the anti-canonical Iitaka dimension.

\begin{theorem}
[cf. Theorem \ref{antiK}] Let $f:X\rightarrow Y$ be an algebraic fibre space between normal projective $\bQ$-Gorenstein varieties, and $F$ be a general fibre of $f$. Suppose $X$ has at worst klt singularities, and $-K_X$ is effective with stable base locus $\textbf{B}(-K_X)$ which does not dominate $Y$. Then we have
\begin{align*}
    \kappa(X,-K_X)\leq \kappa(F,-K_F)+\kappa(Y,-K_Y).
\end{align*}
\end{theorem}

Note that the strict inequality in Theorem 1.1 can happen:

\begin{example}
Let $X$ be an elliptic K3 surface, with a morphism $f:X\rightarrow\bP^1$ whose general fibres are smooth elliptic curves. Then $0=\kappa(X,-K_X)< \kappa(F,-K_F)+\kappa(Y,-K_Y)=1$, which shows that $X$ is an example of strict inequality in Theorem 1.1.
\end{example}

The proof of Theorem 1.1 relies on the following theorem:

\begin{theorem}
[cf. Theorem \ref{eff}] Let $f:X\rightarrow Y$ be algebraic fibre space between normal projective varieties such that $Y$ is $\bQ$-Gorenstein. Let $\Delta$ an effective $\bQ$-Weil divisor on $X$ such that $(X,\Delta)$ is klt, and $D$ a $\bQ$-Cartier divisor on $Y$. Suppose that $-(K_X+\Delta)-f^*D$ is $\bQ$-effective with the stable base locus $\textbf{B}(-(K_X+\Delta)-f^*D)$ that does not surject to $Y$. Then $-K_Y-D$ is $\bQ$-effective.
\end{theorem}

The motivation of our investigation traces back to the following question asked by Demailly-Peternell-Schneider in \cite{DPS}:

\begin{question} Let $f:X\rightarrow Y$ be a surjective morphism between normal projective $\bQ$-Gorenstein varieties. Suppose that $-K_X$ is pseudo-effective and the non-nef locus of $-K_X$ does not dominate $Y$. Is $-K_Y$ pseudo-effective?
\end{question}

We recall some previous results towards answering the above question:
\begin{itemize} 
    \item In \cite[Main Theorem]{CZ}, Chen and Zhang proved that if there is a log canonical pair $(X,\Delta)$ such that if $-(K_X+\Delta)$ is nef, then $-K_Y$ is pseudo-effective. Also, \cite[Example 1.5]{CZ} shows that if the non-log canonical locus of $(X,\Delta)$ surjects onto $Y$, then $-(K_X+\Delta)$ being nef does not imply that $-K_Y$ is pseudo-effective. 
    
    \item In \cite[Theorem D]{Den}, by using analytic methods, Deng proved that if there is a log canonical pair $(X,\Delta)$ such that $-(K_X+\Delta)$ is pseudo-effective and the non-nef locus of $-(K_X+\Delta)$ does not surject onto $Y$ via $f$, then $-K_Y$ is pseudo-effective.
    
    \item In \cite[Theorem 3.1]{EG}, Ejiri and Gongyo generalized Chen and Zhang's theorem, by showing that even if $(X,\Delta)$ is not log canonical, as long as $(F,\Delta|_F)$ is lc for general fibres $F$, then $-K_Y$ is still pseudo-effective. The proof is algebraic and can be generalized to the positive characteristic.
\end{itemize}
Using the method in the proof of \cite[Theorem 3.1]{EG}, with additional ideas from \cite[Main Theorem]{CZ}, we can give an algebraic proof of Theorem \ref{pe}, which generalized \cite[Theorem 3.1]{EG} in characteristic zero. Furthermore, Theorem \ref{pe} simplifies in the following theorem which generalizes \cite[Main Theorem]{CZ}. 

\begin{theorem}
[cf. Theorem \ref{pe}] Let $f:X\rightarrow Y$ be an algebraic fibre space between normal projective varieties. Suppose the following conditions hold:
\begin{enumerate}
    \item There is a log pair $(X,\Delta)$ which is log canonical;
    
    \item $Y$ is $\bQ$-Gorenstein and there is a $\bQ$-Cartier divisor $D$ on $Y$ such that $L:=-(K_X+\Delta)-f^*D$ is a pseudo-effective $\bQ$-Cartier divisor;
    
    \item The restricted base locus $\textbf{B}_-(L)$ does not surject onto $Y$ via $f$.
\end{enumerate}
Then $-K_Y-D$ is pseudo-effective.
\end{theorem}
Note that our result does not cover \cite[Theorem D(a)]{Den}, since for varieties with singularities worse than klt singularities, it is not confirmed if ${\rm NNef}(-)=\textbf{B}_-(-)$ for $\bQ$-Cartier divisors (cf. \cite[Conjecture 1.7]{BBP}, \cite[Theorem 1.2]{CDiB}). Using this theorem, we can give an algebraic proof of a bigness criterion of anti-canonical divisor:

\begin{theorem}
[cf. Theorem \ref{big}] Under the same notation and assumption of Theorem 1.5. Assume $L$ is big, and one of the following conditions holds:
\begin{enumerate}
    \item $\textbf{B}_+(L)$ does not dominate $Y$;
    \item $\textbf{B}_-(L)$ does not surject onto $Y$, and $(X,\Delta)$ is klt.
\end{enumerate}
Then $-K_Y-D$ is big.
\end{theorem}

In the above theorems, the assumption that the asymptotic base locus does not surject onto $Y$ is essential as the following example will show.

\begin{example} Let $Y$ be a smooth curve with genus $g\geq 2$, and  consider the ruled surface $X=\bP_Y(\cO\oplus\cO(-K_Y-D))$, where $D$ is an ample divisor on $Y$ with $\deg D=d>2g-2$. Then $K_X = -2C_0-f^*D$, where $f$ is the structure morphism $f:X\rightarrow Y$ and $C_0$ is the distinguished section. Note that $-K_X\geq f^*D+C_0\geq 0$ and $(f^*D+C_0)^2=d+2-2g>0$, which implies $-K_X$ is big, but $-K_Y$ is anti-ample. Hence the conclusions of the above theorems fail. Note that in this case, we have $\textbf{B}_-(-K_X)=\textbf{B}(-K_X)=\textbf{B}_+(-K_X)=\Supp C_0$, which is surject onto $Y$.
\end{example}

The structure of this paper is as follows. In Section 2, we give a brief review of asymptotic invariants for $\bQ$-Cartier divisor, klt-trivial fibration, moduli {\bf b}-divisors, and weakly positive sheaves. In Section 3, we generalize \cite[Theorem 3.1]{EG} into Theorem \ref{pe}. Then we apply Theorem \ref{pe} to prove Theorem \ref{big}. This is then followed by an application of the theory of moduli $\bQ$-${\bf b}$-divisor and \cite[Theorem 3.3]{A} to prove Theorem \ref{eff}. This leads to Corollary \ref{rela}, where we describe some asymptotic invariants of the relative anti-canonical divisor $-K_{X/Y}$. In Section 4, we prove Theorem \ref{antiK}, by first using the methodology of \cite{EG} to show the inequality for $\kappa(Y,-K_Y)=0$ then using the Iitaka fibrations to obtain the general case. In section 5, we discuss related questions and possible generalizations.\\

\textbf{Acknowledgements:} I would like to thank my advisor Jungkai Alfred Chen for his valuable advice and encouragement. This work can not be done without his consistent support and enlightening suggestions. I want to thank Yoshinori Gongyo for his inspiring classes during his visit to the National Taiwan University in 2019 and for his help and detailed explanations of his works. I thank Meng Chen and Qi Zhang for answering my questions. I would like to thank David Wen for his helpful comments about writing techniques. Next, I want to thank Yoshinori Gongyo and Sho Ejiri for reading my preprint, this article was inspired by their work in \cite{EG}. Also, I want to thank the reviewer for many useful comments on this paper. Next, I want to thank Marta Benozzo and Iacopo Brivio for pointing out errors in the proof of Theorem \ref{eff} after this article had been published online and for their fruitful comment on this revised version.

I would also like to thank the National Taiwan University, the National Center of Theoretical Sciences, and the Ministry of Science and Technology of Taiwan for their generous support. Finally, I want to thank Ching-Jui Lai, Jheng-Jie Chen, Chih-Wei Chang, Hsin-Ku Chen, Hsueh-Yung Lin, Bin Nguyen, and Shi-Xin Wang for their helpful suggestions and encouragement.

\section{Preliminaries}
In this paper, we work over an algebraically closed field $k$ with characteristic zero. A morphism between normal varieties $f:X\rightarrow Y$ is called an {\em algebraic fibre space} if it is a surjective projective morphism with connected fibres. In particular, if $f$ is an algebraic fibre space, then $f_*\cO_X=\cO_Y$. We say a $\bQ$-divisor $D$ on a variety is $\bQ$-Cartier (resp. $\bQ$-effective) if there is a positive integer $m$ such that $mD$ is a Cartier divisor (resp. linearly equivalent to an effective divisor.) In particular, a $\bQ$-Cartier divisor $D$ is $\bQ$-effective if and only if its Iitaka dimension $\kappa(X,D)\geq 0.$

For a $\bQ$-divisor $D=\sum d_iD_i$ on $X$, we will denote
\begin{align*}
    D^+:=\sum_{d_i>0} d_iD_i,\hspace{1cm}  D^-:=\sum_{d_i<0} d_iD_i,
\end{align*}
and that
\begin{align*}
    D^h:=\sum_{f(\Supp D_i)=Y} d_iD_i,\hspace{1cm}  D^v:=\sum_{f(\Supp D_i)\neq Y} d_iD_i.
\end{align*}
We called $D=D^+-D^-$ the {\em effective decomposition}, and $D=D^h+D^v$ the {\em vertical decomposition} of $D$. For other standard notions of birational geometry, such as singularities of pairs, we refer the readers to \cite{KM98} for more details.

\subsection{Asymptotic invariants of $\bQ$-Cartier divisors.}

In this subsection, we will give a brief review of the basic properties of asymptotic invariants of $\bQ$-divisors. Further details of this topic can be found in \cite{BBP}, \cite{BKKMSU}, and \cite{ELMNP}.

\begin{definition}
Let $D$ be a Cartier divisor on a normal variety, the stable base locus of $D$, denoted by $\textbf{B}(D)$, is defined as follows:
$$\textbf{B}(D):=\cap_{m\in \bZ_{>0}}\textbf{Bs}(mD).$$ If $D$ is a $\bQ$-Cartier divisor, and $l$ is a natural number such that $lD$ is Cartier, then we define $\textbf{B}(D)$ to be the stable base locus of $lD$.

Moreover, by fixing an ample divisor $A$, we can define the following sets
\begin{align*}
    \textbf{B}_+(D)&:=\cap_{\varepsilon>0}\textbf{B}(D-\varepsilon A);\\
    \textbf{B}_-(D)&:=\cup_{\varepsilon>0}\textbf{B}(D+\varepsilon A).
\end{align*}
It is well-known that both $\textbf{B}_+(D)$ and $\textbf{B}_-(D)$ are independent of the choice of $A$.
We call $\textbf{B}_+(D)$ the augmented base locus of $D$, and  $\textbf{B}_-(D)$ the restricted base locus (or diminished base locus) of $D$.
\end{definition}

Clearly, from the definition, we have $$\textbf{B}_-(D)\subset\textbf{B}(D)\subset\textbf{B}_+(D).$$ It is  easy to see that $\textbf{B}_+(D)$ and $\textbf{B}_-(D)$ depend only on the numerical equivalent class of $D$ and hence they are numerical invariants, however ${\bf B}(D)$ is not a numerical invariant (cf. \cite[Example 1.1]{ELMNP}).

\begin{remark}
{\em By the definition, both $\textbf{B}_+(D)$ and $\textbf{B}(D)$ are Zariski closed subsets of $X$. But by \cite{L}, $\textbf{B}_-(D)$ can be a countably infinite union of Zariski closed subsets.}
\end{remark}

From the theory of the cone of divisors (cf. \cite[Chapter 1.4, Chapter 2.2]{La}), we immediately have the following correspondence between these asymptotic invariants and positivity of $\bQ$-divisor (cf. \cite[section 4]{BKKMSU}):
\begin{align*}
    &D\ {\rm is\ pseudo-}{\rm effective} &\Leftrightarrow &\textbf{B}_-(D)\neq X,\\
    &D\ {\rm is\ effective} &\Leftrightarrow &\textbf{B}(D)\neq X,\\
    &D\ {\rm is\ big} &\Leftrightarrow &\textbf{B}_+(D)\neq X,\\
    &D\ {\rm is\ nef} &\Leftrightarrow &\textbf{B}_-(D)=\emptyset,\\
    &D\ {\rm is\ semiample} &\Leftrightarrow &\textbf{B}(D)=\emptyset,\\
    &D\ {\rm is\ ample} &\Leftrightarrow &\textbf{B}_+(D)=\emptyset.
\end{align*}

\subsection{The klt trivial fibration and moduli (b-)divisor.}

In this subsection, we will give a brief review of the klt-trivial fibration and the moduli {\bf b}-divisors. Details of these topics can be found in \cite{A}, \cite{A2}, \cite{F12}, \cite{FG14}, \cite{K97}, and \cite{K98}.

\begin{definition}
Let $D$ be a divisor on a normal variety $X$. A {\bf b}-divisor ${\bf D}$ contains a family of divisors $\{ {\bf D}_{X'}\}$, where $X'$ is taken over all higher birational models $\pi:X'\rightarrow X$ such that $\pi$ is a proper birational morphism, ${\bf D}_{X'}$ is a divisor on $X'$ such that $\pi_*({\bf D}_{X'})=D$, and $\pi_*({\bf D}_{X''})={\bf D}_{X'}$ for any birational morphism $\pi': X''\rightarrow X'$. If $D$ is a $\bQ$-divisor, then the $\bQ$-{\bf b}-divisor is defined in the same way.
\end{definition}

\begin{definition}
A klt-trivial fibration (which is equivalent to the lc-trivial fibration in the sense of \cite{A2}) is an algebraic fibre space $f:(X,B)\rightarrow Y$ between normal varieties with a sub log pair $(X, B)$ such that
\begin{enumerate}
    \item $(X,B)$ is subklt over the generic point of $Y$;
    \item rank $f_*\cO_X(\lceil {\bf A}(X,B)\rceil)=1$;
    \item $K_X+B\sim_{\bQ} f^*D$ for some $\bQ$-Cartier divisor $D$ on $Y$.
\end{enumerate}
Here, the discrepancy $\bQ$-{\bf b}-divisor ${\bf A}(X,B)=\{ {\bf A}_{X'}\}$ is defined by the formula
$$ K_{X'}=\pi^*(K_X+B)+ {\bf A}_{X'}.$$
\end{definition}

Next, we recall the definition of the moduli $\bQ$-{\bf b}-divisors and the discriminant $\bQ$-{\bf b}-divisors.

\begin{definition}
Let $f:(X,B)\rightarrow Y$ be a klt-trivial fibration, we define the discriminant $\bQ$-divisor $B_Y$ of $f:(X,B)\rightarrow Y$ in the following way: Let $P$ be a prime divisor on $Y$, which is Cartier in a neighborhood of its generic point, then we define $$b_P:=\max\{t\in \bQ| (X,B+tf^*P)\ {\rm is\ sublc\ over\ the\ generic\ point\ of\ } P\}, $$
and set $$ B_Y:=\sum_P (1-b_P)P, $$ where $P$ runs over all prime divisor of $Y$. We set $M_Y=D-K_Y-B_Y$ and  call $M_Y$  the moduli $\bQ$-divisor of $f:(X,B)\rightarrow Y$.

The moduli $\bQ$-{\bf b}-divisor ${\bf M}=\{{\bf M}_{Y'}\}$ and the discriminant $\bQ$-{\bf b}-divisor ${\bf B}=\{{\bf B}_{Y'}\}$ is defined in the following way: For a proper birational morphism $\mu:Y'\rightarrow Y$, let $X'$ be a normalization of the main component of $X \times_{Y} Y'$ such that the induced morphism $\pi:X'\rightarrow X$ is proper and birational. Define $B_{X'}$ by $$K_{X'}+B_{X'}=\pi^*(K_X+B),$$ then $f':(X',B_{X'})\rightarrow Y'$ is also a klt-trivial fibration. Thus, let $M_{Y'}$ and $B_{Y'}$ be the moduli $\bQ$-divisor and discriminant $\bQ$-divisor of $f':(X',B_{X'})\rightarrow Y'$. Then we set ${\bf M}_{Y'}=M_{Y'}$ and ${\bf B}_{Y'}= B_{Y'}.$
\end{definition}
By the definition, it is easy to see that if $B$ is $\bQ$-effective, then so is $B_Y$.

\subsection{Weakly positive sheaves.}

In this subsection, we will give a brief review of the definition and basic properties of weakly positive sheaves. We adopt the convention and results in \cite[Section 2.2]{EG} which will be necessary. More details of weakly positive sheaves can be found in \cite[Section 4.2]{C}, \cite[Section 2.2]{EG}, and \cite[Section 1]{V}.

\begin{definition}
Let $X$ be a normal quasi-projective variety, $\cG$ be a coherent sheaf on $X$, and $A$ be a fixed ample divisor on $X$. 
\begin{enumerate}
    \item We say that $\cG$ is generically globally generated if the natural map $H^0(X,\cG)\otimes\cO_X\rightarrow \cG$ is surjective over the generic point of $X$.
    \item We say that  $\cG$ is weakly positive if for any natural number $n$, there is a natural number $m$ such that the sheaf $(S^{nm}(\cG))^{**}\otimes\cO_X(mA)$ is generically globally generated, where $S^n(-)$ denotes the $n$-th symmetric power, and $(-)^{**}$ denotes the double dual.
\end{enumerate}
\end{definition}
Note that the definition is independent of the choice of $A$, and when $\cG=\cO_X(D)$ is a line bundle, $\cG$ is weakly positive if and only if $D$ is a pseudo-effective divisor. Moreover, if $\cG|_U$ is weakly positive on $U$ for some open dense subset $U\subset X$ such that $X-U$ has codimension at least 2, then $\cG$ is weakly positive on $X$.

Here we recall two useful lemmas about the weakly positive sheaves.

\begin{lemma}\label{eg1}
(\cite[Lemma 2.4]{EG}) Let $f:Y'\rightarrow Y$ be a surjective projective morphism between geometrically normal quasi-projective varieties over a field, let $\cG$ be a torsion-free coherent sheaf on $Y$:
\begin{enumerate}
    \item If there is no $f$-exceptional divisor on $Y'$, and $\cG$ is weakly positive, then $f^*\cG$ is also weakly positive. Here a prime divisor $E$ on $Y'$ is called $f$-exceptional if $f(E)$ has codimension at least 2 in $Y$.
    \item If $f^*\cG\otimes \cO_{Y'}(E)$ is weakly positive for some effective $f$-exceptional divisor $E$ on $X$, then $\cG$ is weakly positive.
\end{enumerate}
\end{lemma}
\begin{lemma} \label{eg2}
(\cite[Lemma 2.5]{EG})Let $\cF\rightarrow \cG$ be a generically surjective morphism between coherent sheaves on a normal quasi-projective variety over a field. If $\cF$ is weakly positive, so is $\cG$. 
\end{lemma}

\section{Positivity of the anti-canonical divisor of the image}

We first prove a pseudo-effectiveness criterion of the anti-canonical divisor, which is a generalization of \cite[Theorem 3.1]{EG} in characteristic zero. The proof follows the original argument of Ejiri and Gongyo closely. 

\begin{theorem}\label{pe}
Let $f:X\rightarrow Y$ be an algebraic fibre space between normal projective varieties such that $Y$ is $\bQ$-Gorenstein. Let $\Delta=\Delta^+-\Delta^-$ be a $\bQ$-Weil divisor on $X$ such that $K_X+\Delta$ is $\bQ$-Cartier,  $(F,\Delta^+|_F)$ is log canonical for general fibre $F$ of $f$, and $f(\Supp \Delta^-)\neq Y$. Suppose that there is a $\bQ$-Cartier divisor $D$ on $Y$ such that $L:=-(K_X+\Delta)-f^*D$ is a pseudo-effective $\bQ$-Cartier divisor such that $\textbf{B}_-(L)$ does not surject onto $Y$ via $f$. Then, for $l\in\bN$ such that $l(K_X+\Delta)$ and $l(K_Y+D)$ are Cartier and $l\Delta^-$ is integral, we have that
\begin{align*}
    \cO_X(l(f^*(-K_Y-D)+\Delta^-+B))
\end{align*}
is weakly positive for some effective $f$-exceptional $\bQ$-divisor $B$. Moreover, if $Y$ has at worst canonical singularities, then we can take $B=0$. 

In particular, if $\Delta$ is effective, then $-K_Y-D$ is pseudo-effective by Lemma \ref{eg1}(2).
\end{theorem}
\begin{proof}
The proof is based on the methods in the proof of \cite[Theorem 3.1]{EG}, modified with ideas in \cite[Main Theorem]{CZ}. First, we prove the case where $f$ is equidimensional \footnote{In this case, since the pullback of $\bQ$-Weil divisors is well-defined, we do not need to assume $-K_Y$ is $\bQ$-Cartier.}. Let
\begin{align*}
    \cF:=\cO_X(l(f^*(-K_Y-D)+\Delta^-)),
\end{align*}
and $A$ be an ample divisor on $X$, then by definition, it suffices to show that for all $n\in\bN$, there is some $m\in\bN$ such that the sheaf $(\cF^{\otimes nm}\otimes \cO_X(mlA))|_U$ is weakly positive on some Zariski open subset $U$ of $X$ with ${\rm codim}(X-U)\geq 2$. Since $L$ is pseudo-effective, for any $n\in \bN$, the $\bQ$-Cartier divisor $G_n:=L+\frac{1}{n}A$ is big with $\textbf{B}(G_n)\subset \textbf{B}_-(L).$ Therefore, $\textbf{B}(G_n)$ is a proper Zariski closed subset that does not dominate $Y$. Let $\pi:X'\rightarrow X$ be a birational morphism such that $\pi^{-1}(\Supp \textbf{B}(G_n)\cup\Supp\Delta)\cup \Exc(\pi)$ is a snc divisor. 
Let $f':=f\circ\pi :X'\rightarrow Y$ and write $$K_{X'}+\Delta'=\pi^*(K_X+\Delta)+E,$$ where $\Delta'$ is the proper transform of $\Delta$. The effective decomposition yields $(\Delta')^{\pm}=(\Delta^{\pm})'$ and we can write $E=E^+-E^-$. Let $F$ be a general fibre of $f$ and $F' = \pi^{-1}(F)$, then every component of $E^-|_{F'}$ has coefficient at most 1 since $(F, \Delta|_F)$ is log canonical. Pick $\Gamma_n \sim_\bQ \pi^*G_n$ be a general divisor such that $(F',(\Delta'^++\Gamma_n+E^-)|_{F'})$ is log canonical, then by \cite[Theorem 4.13]{C}, for $k\gg 0$ the sheaf $$f'_*\cO_{X'}(kl(K_{X'}+\Delta'^+ +\Gamma_n+E^-))\otimes\cO_Y(-klK_Y)$$ is locally free and weakly positive on a smooth open subvariety $Y_0\subset Y$ with codim$(Y-Y_0)\geq 2$, hence the sheaf is weakly positive on $Y$. Therefore, for $m\gg 0$, we have the following generically surjective morphisms
\begin{align*}
    &f^*(f'_*\cO_{X'}(nml(K_{X'}+\Delta'^++\Gamma_n+E^-)\otimes\cO_Y(-nmlK_Y))\\
    \cong &f^*f_*\pi_*\cO_{X'}(nml(K_{X'}+\Delta'^++\Gamma_n+E^-))\otimes f^*\cO_Y(-nmlK_Y)\\
    \rightarrow&\pi_*\cO_{X'}(nml(K_{X'}+\Delta'^++\Gamma_n+E^-))\otimes f^*\cO_Y(-nmlK_Y)\\
    =&\pi_*\cO_{X'}(nml(E^++\Delta'^--\pi^*f^*D+\frac{1}{n}\pi^*A))\otimes f^*\cO_Y(-nmlK_Y)=:\cG.
\end{align*}
By \cite[Theorem 1.7.6]{La}, the map $f^*f_*(-)\rightarrow (-)$ is generically surjective since on the open dense subset $X-(\Supp(\Delta^-)\cup\Supp(f^*D))$, the sheaf $$\pi_*\cO_{X'}(nml(K_{X'}+\Delta'^++\Gamma_n+E^-))=\pi_*\cO_{X'}(nml(E^++\Delta'^--\pi^*f^*D+\frac{1}{n}\pi^*A))$$ is isomorphic to the ample line bundle $\cO_X(mlA),$ and $\Supp(\Delta^-)\cup\Supp(f^*D)$ does not dominate $Y$. As $f$ is equidimensional, there is no $f$-exceptional divisor, thus $\cG$ is weakly positive by Lemma \ref{eg1}(1) and Lemma \ref{eg2}.

Let $\cG_1: =(\pi_*\cO_{X'}(nml(\Delta'^--\pi^*f^*D+\frac{1}{n}\pi^*A))\otimes f^*\cO_Y(-nmlK_Y))$.
Since $E^+$ is effective and $\pi$-exceptional, there is a smooth open subvariety $U\subset X$ with ${\rm codim} (X-U, X) \ge 2$  such that $\cF$ and $\cG$ are locally free on $U$ and
$\cG|_U \cong \cG_1|_U$.

Therefore, $\cG_1$ is weakly positive since ${\rm codim}(X-U, X)\geq 2$. By the projection formula, we have  
\begin{align*}
  \cG_1|_U  \cong&(\cO_{X}(nml(\Delta^--f^*D+\frac{1}{n}A))\otimes f^*\cO_Y(-nmlK_Y))|_U\\
    \cong&(\cO_{X}(nml(\Delta^--f^*(K_Y+D)))\otimes\cO_X(mlA))|_U\\
    \cong&(\cF^{\otimes nm}\otimes \cO_X(mlA))|_U,
\end{align*}  
where the last isomorphism is due to $\cF|_U$ being a line bundle by our choice of $U$. $(\cF^{\otimes nm}\otimes \cO_X(mlA))|_U$ is weakly positive on $U$. It follows that $\cF^{\otimes nm}\otimes \cO_X(mlA)$ is weakly positive on $X$ since ${\rm codim}(X-U, X)\geq 2$. Lastly, $n$ can be taken to be arbitrarily large, thus $\cF$ is weakly positive.

The proof for the general case follows verbatim to the argument of \cite[Theorem 3.1]{EG}. For the convenience of the readers, we reproduce the proof here. 

By the flattening theorem in \cite[3.3, flattening lemma]{AO}, there is a normal birational modification $\mu:Y'\rightarrow Y$ such that let and $X'$ be the normalization of the main component of $X\times_{Y} Y'$, then $\pi:X'\rightarrow X$ is proper birational and the induced morphism $f':X'\rightarrow Y'$ is equidimensional\\
\begin{center}
    \begin{tikzcd}
&X'\arrow[r,"\pi"]\arrow[d,"f'"] &X\arrow[d,"f"]\\
&Y'\arrow[r,"\mu"] &Y.
\end{tikzcd}
\end{center}
Now we define $\Delta'$ by $$K_{X'}+\Delta'=\pi^*(K_X+\Delta),$$ and write $\Delta'=\Delta'^+-\Delta'^-$, then we have $$-(K_{X'}+\Delta')-\pi^*f^*D=\pi^*(-(K_X+\Delta)-f^*D)$$ is pseudo-effective with the restricted base locus does not surject onto $Y'$. Therefore, $\cO_{X'}(l(f^*(-K_Y'-\mu^*D)+\Delta'^-))$ is weakly positive since $f'$ is equidimensional. Now write $K_{Y'}=\mu^*K_Y+E$, and $E=E^+-E^-$, then we have $$\mu^*(-K_Y)+E^-=-K_{Y'}+E^+\geq -K_{Y'}.$$ Thus $\cO_{X'}(l(f'^*(\mu^*(-K_Y)+E^--\mu^*D)+\Delta'^-))$ is also weakly positive, and so does $\cO_{X}(l(f^*(-K_Y-D)+\Delta^-+\pi_*(f'^*E^-)))$. Now, define $B:=\pi_*(f'^*E^-)$, then since $f'$ is equidimensional, we have $f_*B=\mu_*E^-=0$. Also, if $Y$ has at worst canonical singularities, then $E^-=0$. This completes the proof. 
\end{proof}

By Theorem \ref{pe}, we generalize \cite[Corollary 2.3]{CZ}.

\begin{corollary}
Let $f:X\rightarrow Y$ be an algebraic fibre space between normal projective varieties such that $Y$ is $\bQ$-Gorenstein. Suppose there exists a log pair $(X,\Delta)$ such that $-(K_X+\Delta)$ is pseudo-effective and $\bQ$-Cartier with the restricted base locus $\textbf{B}_-(-(K_X+\Delta))$ not surject onto $Y$, and for general fibre $F$ of $f$, $(F,\Delta|_F)$ is log canonical. Then either $Y$ is uniruled, or $K_Y\sim_{\bQ} 0$.
\end{corollary}
\begin{proof}
The proof follows from Theorem \ref{pe} by repeating the original proof in \cite[Corollary 2.3]{CZ}. Suppose $Y$ is not uniruled, then $K_Y$ is pseudo-effective by \cite[Theorem 2.6]{BDPP}. However, by Theorem \ref{pe}, $-K_Y$ is pseudo-effective. This implies that $K_Y\sim_{\bQ} 0$ by \cite[Corollary 4.9]{N}.
\end{proof}

\begin{corollary}
Let $f:(X, \Delta) \rightarrow Y$ be an algebraic fibre space between normal projective varieties such that $(X,\Delta)$ is klt and  $Y$ is $\bQ$-Gorenstein. Suppose furthermore that $-(K_X+\Delta)$ is pseudo-effective with the non-nef locus ${\rm NNef}(-(K_X+\Delta))$ not surject onto $Y$, then $-K_Y$ is pseudo-effective.
\end{corollary}
\begin{proof}
It directly follows from \cite[Theorem 1.2]{CDiB} and Theorem \ref{pe}.
\end{proof}

\begin{corollary}\label{CEG}
The following theorems \cite[Theorem 4.2, Proposition 4.4, Corollary 4.5, Corollary 4.7, and Corollary 5.1]{EG} still hold under the assumption $L$ is pseudo-effective with $\textbf{B}_-(L)$ not surject onto $Y$ (instead of being nef).
\end{corollary}
\begin{proof}
Pseudo-effectiveness and $\textbf{B}_-(-)$ depends on the numerical equivalence classes only. So, in the proofs of all the above-mentioned theorems in \cite{EG}, we can replace the nefness of $L$ in \cite{EG} by $L$ being pseudo-effective with $\textbf{B}_-(L)$ not surject onto $Y$. Using Theorem \ref{pe} instead of \cite[Theorem 3.1]{EG}, then the same conclusions hold.
\end{proof}

As a corollary of Theorem \ref{pe}, we have a bigness criterion for $-K_Y-D$, whose statement is very similar to \cite[Theorem 3.1]{FG12}, but they cover different cases.

\begin{corollary}\label{FG0}
Keep the same notation as in Theorem \ref{pe}. Suppose there is a big $\bQ$-Cartier divisor $H$ on $Y$ such that $-(K_X+\Delta)-f^*(D+H)$ is pseudo-effective with ${\bf B}_-(-(K_X+\Delta)-f^*(D+H))$ not surject onto $Y$. Then for any $\bQ$-Cartier divisor $D_0$ on $Y$, there is a rational number $\varepsilon>0$ such that for sufficiently divisible positive integer $l$, the sheaf $$\cO_X(l(f^*(-K_Y-D-\varepsilon D_0)+\Delta^-+B))$$ is weakly positive for some effective  $f$-exceptional $\bQ$-divisor $B$.  We can take $B=0$ if $Y$ has at worst canonical singularities. 
\end{corollary}
\begin{proof}
This statement follows immediately from Theorem \ref{pe}.
\end{proof}

\begin{remark}
{\em In Corollary \ref{FG0}, if $\Delta\geq 0$, this means for any $\bQ$-Cartier divisor $D_0$ on $Y$, there is a rational number $\varepsilon>0$ such that $-K_Y-D-\varepsilon D_0$ is pseudo-effective, hence $-K_Y-D\in{\rm Int(Eff}(Y))={\rm Big}(Y)$ is big. }
\end{remark}

As another application of Theorem \ref{pe}, we have another bigness criterion of $-K_Y$ in the following. The reader can compare our results with \cite[Theorem E]{Den} and \cite[Corollary 3.5 and Remark 3.6]{EIM}.

\begin{theorem}\label{big}
Under the same notation and assumption as in Theorem \ref{pe}. Suppose furthermore that  $L:=-(K_X+\Delta)-f^*D$ is big, and one of the following holds:
\begin{enumerate}
    \item $\textbf{B}_+(L)$ does not dominate $Y$;
    \item $\textbf{B}_-(L)$ does not surject onto $Y$, and $(F,\Delta^+_F)$ is klt.
\end{enumerate}
Then for any $\bQ$-Cartier divisor $D_0$ on $Y$, there is a rational number $\varepsilon>0$ such that for sufficiently large and divisible integer $l$, the sheaf $$\cO_X(l(f^*(-K_Y-D-\varepsilon D_0)+\Delta^-+B))$$ is weakly positive for some effective $f$-exceptional $\bQ$-divisor $B$. We can take $B=0$ if $Y$ has at worst canonical singularities. 

In particular, if $\Delta\geq 0$, then $-K_Y-D$ is big.
\end{theorem}
\begin{proof}
For case (1), let $A$ be an ample divisor on $X$ such that $A-f^*D_0$ is ample and effective. Since $\textbf{B}_+(L)$ does not surject onto $Y$, there exists $0<\varepsilon\ll 1$ such that $\textbf{B}(L-\varepsilon A)$ does  not surject onto $Y$. We have
\begin{align*}
    \textbf{B}(L-\varepsilon A)
    \supseteq \textbf{B}(L-\varepsilon A+\varepsilon(A-f^*D_0))=\textbf{B}(L-\varepsilon f^*D_0)
    \supseteq \textbf{B}_-(L-\varepsilon f^*D_0).
\end{align*}
Therefore,  $\textbf{B}_-(L-\varepsilon f^*D_0)$ does not surject onto $Y$.  By Theorem \ref{pe}, we conclude that the sheaf $\cO_X(l(f^*(-K_Y-D-\varepsilon D_0)+\Delta^-+B))$ is weakly positive for sufficiently divisible integer $l$ and some effective $f$-exceptional $\bQ$-divisor $B$.

For (2), it suffices to show that $\textbf{B}_-(-(K_X+\Delta')-f^*D-\varepsilon f^*D_0)$ does not surject over $Y$ for some $\bQ$-divisor $\Delta'$ where $(F,\Delta'^+_F)$ is log canonical. 

We consider $L_{n}:=L-\frac{1}{n} f^*D_0.$ Note that $L$ is big and hence so is  $L_n$  for $n \gg 0$. 
Let $G$ be an effective $\bQ$-divisor such that $L_n-G$ is ample. Since we assume that $\textbf{B}_-(L)$ does not surject onto $Y$, it follows that $\textbf{B}(L+\delta A)$ does not surject onto $Y$ for any $\delta>0$ and ample divisor $A$. 
Hence for any $p\in \bN$, 
$$\textbf{B}(pL-\frac{1}{n} f^*D_0+\delta A -G)  = \textbf{B}((p-1)L+\delta A+L_n -G) \subset \textbf{B}((p-1)L+\delta A)$$ does not surject onto $Y$.  
As a result, neither does $\textbf{B}_-(pL-\frac{1}{n} f^*D_0-G)=\textbf{B}_-(L_{pn}-\frac{1}{p}G)$. 

Now since $(F,\Delta^+_F)$ is klt,   by letting $\Delta':=\Delta+\frac{1}{p}G$, we have $(F,\Delta'^+_F)$ is still klt for $p\gg0$. 
Thus,
\begin{align*}
    \textbf{B}_-(L_{pn}-\frac{1}{p}G)=\textbf{B}_-(-(K_X+\Delta')-f^*D-\frac{1}{np} f^*D_0)
\end{align*} does not surject onto $Y$. By Theorem \ref{pe}, there is an effective $f$-exceptional $\bQ$-divisor $B$ on $X$ such that $\cO_X(l(f^*(-(K_Y+D)-\frac{1}{np}D_0)+\Delta^-+B))$ is weakly positive for sufficiently divisible integer $l$.
\end{proof}

Observing Theorem \ref{pe} and Theorem \ref{big}(1), it is reasonable to expect  analogous statements under the assumption that $-K_X$ is $\bQ$-effective.

\begin{theorem} \label{eff}
Let $f:X\rightarrow Y$ be an algebraic fibre space between normal projective varieties such that $Y$ is $\bQ$-Gorenstein. Let $\Delta=\Delta^+-\Delta^-$ be a $\bQ$-divisor on $X$ such that $(K_X+\Delta)$ is $\bQ$-Cartier, $f(\Supp\Delta^-)\neq Y$ and $(F,\Delta^+_F)$ is klt for general fibres $F$ of $f$. Let $D$ be a $\bQ$-Cartier divisor on $Y$ such that $L:=-(K_X+\Delta)-f^*D$ is $\bQ$-effective with the stable base locus $\textbf{B}(L)$ does not dominate $Y$, then we have that $f^*(-K_Y-D)+\Delta^-+E_X$ is $\bQ$-effective for some effective $f$-exceptional $\bQ$-divisor $E_X$. Moreover, we can take $E_X=0$ if one of the following assumptions holds:
\begin{enumerate}
\item $f$ is equidimensional;
\item $\Delta$ is effective;
\item $Y$ has at worst canonical singularities.
\end{enumerate}
\end{theorem}
\begin{proof}
The idea is similar to the proof of \cite[Theorem 4.1]{A2}. At first, we prove that $f:(X,B)\rightarrow Y$ is a klt-trivial fibration in the sense of \cite{FG14}, where $B:=\Delta+L$. Since $L$ is $\bQ$-effective with stable base locus $\textbf{B}(L)$ which does not dominate $Y$, then up to $\bQ$-linear equivalence, we may assume that $\Delta_F+L|_F$ is effective and klt for general fibres $F$. Let $B:=\Delta+L$, then by the above discussion, we have ${\rm Nklt}(\Delta^++L)$ does not surject onto $Y$, so $(X, B)$ is subklt over the generic point of $Y$. 

Let $\pi:X'\rightarrow X$ be any birational morphism such that $X'$ is smooth. By the formula $ K_{X'}=\pi^*(K_X+B)+ {\bf A}_{X'},$ we have ${\bf A}_{X'}= -B'+E$, where $B'$ is the proper transform of $B$ on $X'$ and $E$ is a $\pi$-exceptional divisor. Since $(X,B)$ is subklt over the generic point of $Y$ and $B^h$ is effective by our assumption, we have $\lceil -B \rceil$ (resp. $\lceil -B' \rceil$) has no horizontal component for $f$ (resp. $f\circ\pi$), and the horizontal component of $\lceil E\rceil$ for $f \circ \pi$ is effective. This implies that $f_*\pi_*\cO_{X'}({\bf A}_{X'})$ is of rank at least 1. Moreover, $\lceil {\bf A}_{X'} \rceil = \lceil -B'+E \rceil = \lceil -B'\rceil+\lceil E \rceil$ is the sum of the vertical divisor $\lceil -B'\rceil$ and the $\pi$-exceptional divisor $\lceil E\rceil$. Now we let $A$ be a Cartier divisor on $Y$ such that $\pi^*f^*A\geq \lceil -B' \rceil$ (such $A$ exists since $\lceil -B'\rceil$ is vertical for $f\circ\pi$), then by the projection formula and \cite[Lemma 7.11 and 7.12]{Deb}, we have 
\begin{align*}
f_*\pi_*\cO_{X'}({\bf A}_{X'})&\leq f_*\pi_*\cO_{X'}(\pi^*f^*A+\lceil E^+\rceil)\\
&= f_*(\cO_X(f^*A)\otimes\pi_*\cO_{X'}(\lceil E^+ \rceil))\\
&\cong f_*(\cO_X(f^*A)\otimes \cO_X)=\cO_Y(A).
\end{align*}
In particular, $f_*\pi_*\cO_{X'}({\bf A}_{X'})$ is of rank 1 for any birational morphism $\pi:X'\rightarrow X$ with $X'$ smooth, hence $f_*\cO_X(\lceil {\bf A}(X,B)\rceil)$ is of rank 1. Moreover, we have $$K_X+B=K_X+\Delta+L\sim_{\bQ} f^*(-D).$$ Therefore, $f:(X,B)\rightarrow Y$ is a klt-trivial fibration in the sense of \cite{FG14}, which is also equivalent to the lc-trivial fibration in the sense of \cite{A2}, so we can apply the construction of moduli divisor on $f':X'\rightarrow Y'$ to get
\begin{align*}
    f^*(-D)\sim_{\bQ}K_{X}+B\sim_{\bQ} f^*(K_{Y}+M_{Y}+B_{Y}),
\end{align*}
and hence $(-K_Y-D)\sim_{\bQ} M_{Y}+B_{Y}$. 

Thus, to show the $\bQ$-effectiveness of $f^*(-K_Y-D)+\Delta^-+E$ for some $f$-exceptional divisor $E$, it remains to show that $f^*(M_{Y}+B_{Y})+\Delta^-+E$ is $\bQ$-effective. Similar to Theorem \ref{pe}, we first prove the theorem under the assumption that $f$ is equi-dimensional. Since $B^h$ is effective, we can apply \cite[Theorem 3.3]{A2} to conclude that the moduli-\textbf{b}-divisor \textbf{M} is nef and abundant. There is a resolution $\mu:Y'\rightarrow Y$ such that $M_{Y'}$ is nef  $\bQ$-effective.  Note that in general $M_Y$ may not be $\bQ$-Cartier, but it is $\bQ$-Cartier over $Y_0$, where $Y_0$ is the smooth locus of $Y$. Since $\mu_*(M_{Y'})=M_Y$, over the open dense subset $\mu^{-1}(Y_0)$ of $Y'$ we have the equality $\mu^*M_Y= M_{Y'}+E^+-E^-$ for $E^+, E^-$ be some effective $\mu$-exceptional $\bQ$-divisors on $Y'$. Hence $$\mu^*M_Y+E^- \geq M_{Y'}$$ is $\bQ$-effective over $\mu^{-1}(Y_0)$. Thus, $\mu^*M_Y$ is $\bQ$-linear equivalent to a divisor which is effective outside the exceptional set over $\mu^{-1}(Y_0)$. It follows that $M_{Y}$ is $\bQ$-effective over $Y_0$. Since $Y$ is normal, we have $Y-Y_0$ has codimension at least 2, which implies $M_Y$ is $\bQ$-effective over $Y$. Next, let $B^v$ be the vertical part of $B$. By the definition of $B$, we have $B+\Delta^-=\Delta^++L$ is effective. Hence $B^v+\Delta^-=(\Delta^++L)^v$ is effective since $\Delta^-$ is vertical. Also, by the definition of $B_Y$, the negative part of $f^*B_Y-B^v$ is $f$-exceptional. Since $f$ is equidimensional, there is no $f$-exceptional divisor on $X$ and hence $f^*B_Y\geq B^v$. Thus, we have $f^*(M_Y+B_Y)\geq B^v$ since $M_Y$ is an effective $\bQ$-divisor, which implies $f^*(M_Y+B_Y)+\Delta^-\geq B^v+\Delta^-\geq 0$.

The proof for the general case again follows verbatim to the argument of \cite[Theorem 3.1]{EG}. By the flattening theorem, there is a normal birational modification $\mu:Y'\rightarrow Y$ such that let $X'$ be the normalization of the main component of $X\times_{Y} Y'$, then $\pi:X'\rightarrow X$ is proper birational and the induced morphism $f':X'\rightarrow Y'$ is equidimensional.
\begin{center}
    \begin{tikzcd}
&X'\arrow[r,"\pi"]\arrow[d,"f'"] &X\arrow[d,"f"]\\
&Y'\arrow[r,"\mu"] &Y
\end{tikzcd}
\end{center}
Now we define $\Delta'$ by $$K_{X'}+\Delta'=\pi^*(K_X+\Delta),$$ and write $\Delta'=\Delta'^+-\Delta'^-$, then we have $$-(K_{X'}+\Delta')-\pi^*f^*D=\pi^*(-(K_X+\Delta)-f^*D)$$ is $\bQ$-effective with the stable base locus not dominant over $Y'$. Therefore, $f^*(-K_{Y'}-\mu^*D)+\Delta'^-$ is $\bQ$-effective since $f'$ is equidimensional. Now write $K_{Y'}=\mu^*K_Y+E$, and $E=E^+-E^-$, then we have $$\mu^*(-K_Y)+E^-=-K_{Y'}+E^+\geq -K_{Y'}.$$ Thus $f'^*(\mu^*(-K_Y)+E^--\mu^*D)+\Delta'^-$ is also $\bQ$-effective, and so does $$\pi_*(f'^*(\mu^*(-K_Y)+E^--\mu^*D)+\Delta'^-)=f^*(-K_Y-D)+\Delta^-+\pi_*(f'^*E^-).$$ Now, define $E_X:=\pi_*(f'^*E^-)$, then since $f'$ is equidimensional, we have $f_*(E_X)=\mu_*E^-=0$, hence $E_X$ is $f$-exceptional. Also, if $Y$ has at worst canonical singularities, then $E^-=0$, hence $E_X=0$. Finally, if $\Delta$ is effective, then $\Delta^-=0$ and hence the $\bQ$-effectiveness of $f^*(-K_Y-D)+E_X$ implies $f^*(-K_Y-D)$ is $\bQ$-effective by Lemma \ref{effexc} (no matter $Y$ has at worst canonical singularities or not).
\end{proof}

By the above theorems in this section, we can give more properties of the relative anti-canonical divisor, generalizing \cite[Theorem 3.12]{Deb} and \cite[Corollary 3.7]{EIM}.

\begin{corollary}\label{rela}
Let $f:X\rightarrow Y$ be an algebraic fibre space between normal projective varieties such that $Y$ is $\bQ$-Gorenstein and is not a point. Let $(X,\Delta)$ be a log pair such that $(F,\Delta_F)$ is log canonical for general fibres $F$ of $f$, then:
\begin{enumerate}
    \item Let $D_0$ be a pseudo-effective $\bQ$-Cartier divisor on $Y$ that is not numerically trivial. Then $\textbf{B}_-(-(K_{X/Y}+\Delta+f^*D_0))$ and $\textbf{B}(-(K_{X/Y}+\Delta+f^*D_0))$ surjects onto $Y$.
    \item $\textbf{B}_+(-(K_{X/Y}+\Delta))$ always surjects onto $Y$. Moreover, if $-(K_{X/Y}+\Delta)$ is big, and $(F,\Delta|_F)$ is klt for general fibres $F$. Then $\textbf{B}_-(-(K_{X/Y}+\Delta))$ and  $\textbf{B}(-(K_{X/Y}+\Delta))$ surjects onto $Y$.
\end{enumerate}
\end{corollary}
\begin{proof}
For (1), let $D= -K_Y+D_0$, then since $-K_Y-D=-D_0$ is not pseudo-effective, $-(K_X+\Delta)-f^*D=-(K_{X/Y}+\Delta+f^*D_0)$ must {\em not} satisfies the condition of Theorem \ref{pe}. Hence $\textbf{B}_-(-(K_{X/Y}+\Delta+f^*D_0))$ should surject onto $Y$. For (2), let $D= -K_Y$ in Theorem \ref{big}, then $-K_Y-D=0$ is not big, so we can get the result by the same method.
\end{proof}

\section{The anti-canonical Iitaka dimension.}

Based on Theorem \ref{eff}, we have the following result, which can be thought of as the anti-canonical version of the Iitaka conjecture.

\begin{theorem}
\label{antiK}
Let $f:(X, \Delta) \rightarrow Y$ be an algebraic fibre space between normal projective varieties such that $\Delta$ is effective, $K_X+\Delta$ is $\bQ$-Cartier, and $Y$ is $\bQ$-Gorenstein. Suppose there is a $\bQ$-Cartier divisor $D$ on $Y$ such that $L:=-(K_X+\Delta)-f^*D$ is $\bQ$-effective and $\textbf{B}(L)$ does not dominate  $Y$. Suppose furthermore that $(F,\Delta_F)$ is klt for general fibres $F$ of $f$, where $\Delta_F$ is defined by $-(K_F+\Delta_F)=-(K_X+\Delta)|_F=L_F$.  Then we have
\begin{align*}
    \kappa(X,L)\leq \kappa(F, L_F)+\kappa(Y,-K_Y-D).
\end{align*}

In particular, if $X$ has at worst klt $\bQ$-Gorenstein singularities, and $-K_X$ is effective with stable base locus $\textbf{B}(-K_X)$ does not dominate $Y$, then we have
\begin{align*}
    \kappa(X,-K_X)\leq \kappa(F,-K_F)+\kappa(Y,-K_Y).
\end{align*}
\end{theorem}

To prove Theorem \ref{antiK},  we need the following proposition, which is a variant of \cite[Proposition 4.4]{EG}.

\begin{proposition}\label{egeff}
Let $f:X\rightarrow Y$ be an algebraic fibre space between normal projective varieties such that $Y$ is $\bQ$-Gorenstein. Let $\Delta=\Delta^+-\Delta^-$ be a $\bQ$-divisor on $X$ such that $(K_X+\Delta)$ is $\bQ$-Cartier, $f(\Supp\Delta^-)\neq Y$ and $(F,\Delta^+_F)$ is klt for general fibres $F$ of $f$. Let $D$ and $E$ be $\bQ$-Cartier divisors on $Y$ such that $L:=-(K_X+\Delta)-f^*D$ is $\bQ$-effective with the stable base locus $\textbf{B}(L)$ does not dominate $Y$. Suppose furthermore that there is an effective $\bQ$-divisor $\Gamma$ such that $L-g^*E\sim_{\bQ} \Gamma \ge 0$. Moreover, assume one of the following three conditions holds:
\begin{enumerate}
\item $f$ is equidimensional;
\item $\Delta$ is effective;
\item $Y$ has at worst canonical singularities.
\end{enumerate}
Then for $0<\varepsilon\ll 1$, $f^*(-K_Y-D-\varepsilon E)+\Delta^-$ is $\bQ$-effective.
\end{proposition}
\begin{proof}
The proof follows the argument of \cite[Proposition 4.4]{EG} closely. We consider $$\Delta_{\varepsilon}:=\Delta+\varepsilon\Gamma;\hspace{1cm} D_{\varepsilon}:=D+\varepsilon E;\hspace{1cm} L_{\varepsilon}:=-(K_X+\Delta_{\varepsilon})-f^*D_{\varepsilon}.$$
Then 
$  L_{\varepsilon} =L-\varepsilon(\Gamma+f^*E)\sim_{\bQ} (1-\varepsilon)L$.
Thus, for $\varepsilon<1$, $L_{\varepsilon}$ is $\bQ$-effective with the stable base locus $\textbf{B}(L_{\varepsilon})$ that does not dominate $Y$. For $0<\varepsilon  \ll 1$,  $(F,(\Delta_{\varepsilon})|_F)$ is still klt for general fibre $F$. Therefore, applying Theorem \ref{eff} on $L_\epsilon$, $f^*(-K_Y-D_{\varepsilon})+\Delta_\varepsilon^-$ is $\bQ$-effective, hence so does $f^*(-K_Y-D-\varepsilon E)+\Delta^-\geq f^*(-K_Y-D_{\varepsilon})+\Delta_\varepsilon^-$.
\end{proof}

In fact, it is not difficult to generalize Proposition \ref{egeff} to the situation that $f$ is an almost holomorphic fibration, by following the same proof in \cite[Proposition 4.4]{EG}. We leave the details to the readers.

The next theorem is a variant of Ejiri and Gongyo's injectivity theorem \cite[Theorem 1.2]{EG}. This is an application of Proposition \ref{egeff}, and plays a key role in the proof of Theorem \ref{antiK}.

\begin{theorem}\label{kd0}
Use the notation and assumption in Proposition \ref{egeff}. Suppose there exists a $\bQ$-Cartier divisor $P$ on $X$ such that $P\geq \Delta^-$, $f(\Supp P)\neq Y$, and $\kappa(X,f^*(-K_Y-D)+P)=0$. Then for any general fibre $F$ of $f$, the following morphism between graded ring defined by restriction
\begin{align*}
    \bigoplus_{m\geq 0}H^0(X,\lfloor mL \rfloor)\rightarrow \bigoplus_{m\geq 0}H^0(F,\lfloor mL_F \rfloor)
\end{align*}
is injective. In particular, this implies $\kappa(X,L)\leq \kappa(F,L_F).$ 
\end{theorem}
\begin{proof}
We closely follow the proof of \cite[Theorem 4.2, Corollary 4.7]{EG}. Consider the map between graded rings defined by restriction
\begin{align*}
    \bigoplus_{m\geq 0}H^0(X,\lfloor mL \rfloor)\rightarrow \bigoplus_{m\geq 0}H^0(F,\lfloor mL_F \rfloor).
\end{align*}
As in the proof of \cite[Corollary 4.7]{EG}, the kernel of this map are the sections corresponding to the effective divisors $N\in|mL|$ such that $\Supp N\supset F$. So it suffices to show that for every effective $\bQ$-divisor $N\sim_{\bQ} L$, $\Supp N$ does not contains $F$.

Since $\kappa(X,f^*(-K_Y-D)+P)=0$, there is a unique effective $\bQ$-Cartier divisor $M\sim_{\bQ}f^*(-K_Y-D)+P$. Let $F$ be a normal irreducible fibre of $f$ such that $y=f(F)$ is a smooth point on $Y$, $f$ is flat over an open dense neighborhood of $y$, $F\cap \Supp P=\emptyset$, and $F\nsubseteq \Supp M$. 
Now define $\pi:X'\rightarrow X$ (resp. $\mu:Y'\rightarrow Y$) to be the blow-up of $X$ (resp. $Y$) with respect to $F$ (resp, $y$). Then as in the proof of \cite[Theorem 4.2]{EG}, since $f$ is flat over an open neighborhood of $y$, we have $X'=X\times_{Y} Y'$ by \cite[Lemma 31.32.3]{S}. Also, since $y$ is a smooth point on $Y$, we have $Y'$ is still normal and $\bQ$-Gorenstein and $X'$ is still normal since $f'$ is flat with normal fibres over a neighborhood of exceptional locus, where $f'$ is the induced morphism $f':X'\rightarrow Y'$. Now we have the following commutative diagram:

\begin{center}
    \begin{tikzcd}
&X'\arrow[r,"\pi"]\arrow[d,"f'"] &X\arrow[d,"f"]\\
&Y'\arrow[r,"\mu"] &Y.
\end{tikzcd}
\end{center}

Write $K_{Y'}=\mu^*K_Y+aE$, where $E$ is the exceptional divisor of $\mu$ on $Y'$, and $a=\dim Y-1$. Let $G$ be the exceptional divisor of $\pi$ (where $G\cong F\times E$), then by the flatness over $F$, we have $G=f'^*E$. Also, since  $F\cap \Supp\Delta^-=\emptyset$, we can write $-\pi^*(K_X+\Delta)=-(K_{X'}+\Delta')+bG$ for some $b \leq {\rm codim}(F,X) - 1 = \dim Y - 1 = a$, where $\Delta'$ is the proper transform of $\Delta$ on $X'$. Let
\begin{align*}
    L':&=-(K_{X'}+\Delta')-f'^*(\mu^*D-aE)\\
    &=-(K_{X'}+\Delta')-\pi^*f^*D+aG
    =\pi^*L+(a-b)G\geq \pi^*L,
\end{align*}
which is also effective with stable base locus that does not dominate $Y'$.

Now, suppose there exists an effective $\bQ$-divisor $N\sim_{\bQ} L$ with $\Supp N\supset F$. Then we have $\Supp(\pi^*N)\supset\Supp G$. Now define $$N':=\pi^*N+(a-b)G\sim_{\bQ} \pi^*L+(a-b)G=L',$$ then $N'$ is effective with $\Supp N'\supset\Supp G$. Therefore, there exists a rational number $0<\delta\ll 1$, such that $\Gamma:=N'-\delta G\geq 0$, then $$L'-\delta f'^*E=L'-\delta G\sim_{\bQ} N'-\delta G=\Gamma\geq 0.$$ So we can apply Proposition \ref{egeff} on $X',Y',\mu^*D-aE,L', \delta E,$ and $\Gamma$ to conclude that $f'^*(-K_{Y'}-\mu^*D+aE-\varepsilon\delta E)+\Delta'^-$ is $\bQ$-effective for $0<\varepsilon\ll 1$, hence $f'^*(-K_{Y'}-\mu^*D+aE-\varepsilon\delta E)+\pi^*P$ is $\bQ$-Cartier and $\bQ$-effective. Note that
\begin{align*}
    f'^*(-K_{Y'}-\mu^*D+aE-\varepsilon\delta E)+\pi^*P&=\pi^*(f^*(-K_Y-D)+P)-\varepsilon\delta G.
\end{align*}
Therefore, $G\subset \Supp(\pi^*M)$ by the uniqueness of $M$, thus $F\subset \Supp M$, but this is a contradiction to our choice of $F$, hence the proof is completed.
\end{proof}

\begin{proof}[Proof of Theorem \ref{antiK}]
First, we prove the theorem under the assumption that $Y$ is smooth. By Theorem \ref{eff}, we have $-K_Y-D$ is $\bQ$-effective, and by Theorem \ref{kd0}, we only need to consider the case of $\kappa(Y,-K_Y-D)>0$.  Consider the following commutative diagram
\begin{center}
    \begin{tikzcd}
&X'\arrow[r,"\pi"]\arrow[d,"f'"] &X\arrow[d,"f"]\\
&Y'\arrow[r,"\mu"]\arrow[d,"g'"] &Y\arrow[dashed,d,"g"]\\
&Z'\arrow[dashed,r,"\nu"] &Z.
\end{tikzcd}
\end{center}
Where
\begin{enumerate}
    \item $g: Y\dashrightarrow Z$ is the rational map defined by $m(-K_Y-D)$ for some sufficiently divisible $m\gg 0$ such that $\dim(\overline{\Img(g)})=\kappa(Y,-K_Y-D)$ and ${\bf Bs}(m(-K_Y-D))={\bf B}(-K_Y-D)$.
    \item $g':Y'\rightarrow Z'$ is a minimal resolution of Iitaka fibration induced by $m(-K_Y-D)$. More precisely, consider the following diagram
    \begin{center}
    \begin{tikzcd}
&Y_0\arrow[r,"\mu_0"]\arrow[d,"g_0"] &Y\arrow[dashed,d,"g"]\\
&Z'\arrow[dashed,r,"\nu"] &Z.
\end{tikzcd}
\end{center}
Here $g_0:Y_0\rightarrow Z'$ is the Iitaka fibration induced by $m(-K_Y-D)$. Note that $Y_0$ is possibly singular, but by the smoothness of $Y$, the singularity of $Y_0$ must contained in $\Exc(\mu_0)$. Now we let $\mu':Y'\rightarrow Y_0$ be a minimal resolution of singularities.
    \item $\mu:= \mu_0\circ \mu':Y'\rightarrow Y$ and $g':=g_0\circ\mu':Y'\dashrightarrow Z'$ are the induced maps, note that $\mu$ is birational and $g'$ is an algebraic fibre space.
    \item $X'$ is the normalization of the main component of $X\times_{Y} Y'$, and $\pi:X'\rightarrow X$ and $f': X'\rightarrow Y'$ are the induced morphism. Note that we have $f'(\Exc (\pi))\subset \Exc(\mu)$. 
    \item $F$ is a general fibre of $f$ (by the construction of $X'$ and $f'$, $F$ is also a general fibre of $f'$).
    \item $G'$ is a general fibre of $g'$. Note that $G'$ is smooth by the smoothness of $Y'$
    \item $W':=f'^{-1}(G')$ is a general fibre of $g'\circ f'$.
\end{enumerate}

Here $X'$ and $f'$ are constructed by the following way: Let $X_0$ be the main component of $X\times_{Y} Y'$ with projections $f_0:X_0\rightarrow Y'$ and $\pi_0:X_0\rightarrow X$. Then we have $\Exc(\pi_0)\subset f_0^{-1}(\Exc(\mu))$ and $\pi_0(\Exc(\pi_0))\subset f^{-1}(\mu(\Exc(\mu)))$. Since $X$ is normal, $X_0-\Exc(\pi_0)\cong X- \pi_0(\Exc(\pi_0))$ is normal. Hence for the normalization $ \pi':X'\rightarrow X_0$ of $X_0$, $\pi'$ is an isomorphism over $X_0-\Exc(\pi_0)$. Thus, we have $$\Exc(\pi)\subset \pi'^{-1}(\Exc(\pi_0)) \subset f'^{-1}(\Exc(\mu)),$$ where $f'= f_0\circ \pi'$ and $\pi:=\pi_0\circ \pi'$.

By the easy addition formula \cite[Corollary 1.7]{Mo}, we have
\begin{align*}
    \kappa(X',\pi^*L)&\leq \kappa(W',(\pi^*L)|_{W'})+\dim Z'.
\end{align*}
Since $\kappa(X',\pi^*L)=\kappa(X,L)$ and $\dim Z'=\kappa(Y,-K_Y-D)$, this implies
\begin{align*}
    \kappa(X,L)\leq \kappa(W',(\pi^*L)|_{W'})+\kappa(Y,-K_Y-D).
\end{align*}
Thus, it remains to show that $\kappa(W',(\pi^*L)|_{W'})\leq \kappa(F,L_F).$ To show this, let $B:=\pi^*(K_{X/Y}+\Delta)-K_{X'/Y'}$, then
\begin{align*}
    &\mu^*(-K_Y-D)=-K_{Y'}-(\mu^*(K_Y+D)-K_{Y'}),\\
    &\pi^*L=-(K_{X'}+B)-f'^*(\mu^*(K_Y+D)-K_{Y'}).
\end{align*}
Next, write
\begin{align*}
    K_{Y'}=\mu^*K_Y+\sum a_iE_i,\  K_{X'}=\pi^*(K_X+\Delta)+\sum b_j P_j,
\end{align*}
then $B=\sum a_if'^*E_i-\sum b_jP_j$. Since $\Delta$ is effective, if $b_j> 0$ then $P_j$ must be $\pi$-exceptional, which implies $f'(P_j)\subseteq\Exc(\mu)$. Therefore, $f'(\Supp B^-)\subset \Exc(\mu)$ and hence $\Supp(B^-)\subset f'^{-1}\Exc(\mu)$. Note that by the construction of the Iitaka fibration, we may assume  $\mu_0(\Exc(\mu_0))$ is contained in ${\bf B}(-K_Y-D)$, which implies $\Exc(\mu_0)\subset \mu_0^{-1}{\bf B}(-K_Y-D)$. Note that by the smoothness of $Y$, we have the singularity of $Y_0$ must contained in $\Exc(\mu_0)$, which implies $\Exc(\mu')\subset \mu'^{-1}(\Exc(\mu_0))$ and hence $$\Exc(\mu)\subset \Exc(\mu') \cup  \mu'^{-1}(\Exc(\mu_0))=\mu'^{-1}(\Exc(\mu_0)) \subset \mu^{-1}({\bf B}(-K_Y-D)).$$ Thus, for sufficiently divisible $m\gg 0$, $\Supp(\mu^*(m(-K_Y-D)))\supset \Exc(\mu)$, which implies $$\Supp(f'^*\mu^*(m(-K_Y-D)))\supset f'^{-1}(\Exc(\mu))\supset \Supp(B^{-}).$$ Therefore, $P_{mn}:=f'^*\mu^*(mn(-K_Y-D))\geq B^-$ for $n\gg 0$. Consider the morphism $f'|_{W'}:W'\rightarrow G',$ whose general fibre is also a general fibre $F$ of $f$ by our construction. Then we have 
\begin{align*}
    \kappa(W',(f'^*\mu^*(-K_Y-D)+P_{mn})|_{W'})&=\kappa(W',f'^*\mu^*(-K_Y-D)|_{W'})\\
    &=\kappa(G',\mu^*(-K_Y-D)|_{G'})\\
    &=\kappa(G_0,\mu_0^*(-K_Y-D)|_{G_0})=0,
\end{align*}
where the last equality is from the structure of Iitaka fibration.

Now, consider the morphism $f'|_{W'}:W'\rightarrow G'$ and the divisors
\begin{align*}
    &-K_{G'}-(\mu^*(K_Y+D)|_{G'}-K_{G'})=\mu^*(-K_Y-D)|_{G'},\\
    &-(K_{W'}+B_{W'})-f'^*(\mu^*(K_Y+D)|_{G'}-K_{G'})=(\pi^*L)|_{W'}.
\end{align*}
Here $B_{W'}$ is defined by $(K_{X'}+B)|_{W'}=K_{W'}+B_{W'}$. Note that we have
\begin{align*}
(K_{W'}+B_{W'})|_F &=(K_{X'}+B)|_F\\
&=(\pi^*(K_X+\Delta)-f'^*(K_{Y'}-\mu^*K_Y))|_F\\
&=(\pi^*(K_X+\Delta))|_F\\
&\cong(K_X+\Delta)|_F\\
&= K_F+\Delta_F,
\end{align*}
where the isomorphism follows from the fact that $\pi$ is an isomorphism over general fibres $F$. In particular, let $B_F$ defined by $(K_{W'}+B_{W'})|_F=K_F+B_F$, then $B_F=\Delta_F$ under the natural isomorphism $\pi^{-1}(F)\cong F$, hence $(F,B_F)$ is klt. Hence we can apply Theorem \ref{kd0} (since $G'$ is smooth) to conclude that
\begin{align*}
    \kappa(W',(\pi^*L)|_{W'})\leq \kappa(F,-(K_{W'}+B_{W'})|_F)
    =\kappa(F,-(K_X+\Delta)|_F)=\kappa(F,L_F).
\end{align*}
This completes our proof when $Y$ is smooth.

For the general case, consider the following diagram
\begin{center}
    \begin{tikzcd}
&X'\arrow[r,"\pi"]\arrow[d,"f'"] &X\arrow[d,"f"]\\
&Y'\arrow[r,"\mu"] &Y.
\end{tikzcd}
\end{center}

Where $Y'$ is a resolution of singularities of $Y$, and $X'$ is a normalization of the main component of $X \times_Y Y'$. Similar to the above argument, we may assume $f'(\Exc (\pi))\subset \Exc(\mu)$. Now we let $L=-(K_X+\Delta)-f^*D$, and write $\pi^*L=-(K_{X'}+\Delta')+E^+-E^--f'^*\mu^*D$, where $\Delta'$ is the proper transform of $\Delta$ on $X'$, and $E^+$, $E^-$ are effective $\pi$-exceptional $\bQ$-divisors defined by the equality $K_{X'}+\Delta'=\pi^*(K_X+\Delta)+E^+-E^-$. Note that by the construction of $X'$, both $E^+$ and $E^-$ are vertical with $f(\Supp E^+), f(\Supp E^-)\subset \Exc(\mu)$, hence there exists an effective $\mu$-exceptional divisor $N$ on $Y'$ such that $f'^*N\geq E^+\geq E^+-E^-$. Now, since the theorem holds if the base is smooth, we have \begin{align*}
\kappa(X,L)=\kappa(X',\pi^*L)&\leq \kappa(X',-(K_{X'}+\Delta')-f^*\mu^*D+f'^*N)\\
&\leq \kappa(F,-(K_F+\Delta_F))+\kappa(Y',-K_{Y'}-\mu^*D+N).
\end{align*}
Note that $-(K_{X'}+\Delta')-f^*\mu^*D+f'^*N=\pi^*L+(f^*N-(E^+-E^-))$ is still effective with the stable base locus does not dominant $Y'$. Now, we write $K_{Y'}=\mu^*K_Y+B^+-B^-$ for some $\mu$-exceptional divisors $B^+,B^-$. Then we have 
\begin{align*}
\kappa(Y',-K_{Y'}-\mu^*D+N)&=\kappa(Y',\mu^*(-K_Y-D)+N+B^--B^+)\\
 &\leq \kappa(Y',\mu^*(-K_Y-D)+N+B^-)\\
 &=\kappa(Y',\mu^*(-K_Y-D))\\
 &= \kappa(Y,-K_Y-D)
\end{align*}
since $N$ and $B$ are $\mu$-exceptional, hence the result holds for general case.
\end{proof}

The following corollaries directly follow from Theorem \ref{antiK}:

\begin{corollary}
Let $(X,\Delta)$ be a log pair on $X$ and $f:X\rightarrow Y$ be an algebraic fibre space between normal projective varieties. Suppose $Y$ is $\bQ$-Gorenstein and $(F,\Delta_F)$ is klt for general fibres $F$. Then:
\begin{enumerate}
    \item Let $D$ be a $\bQ$-Cartier divisor on $X$ such that $-(K_X+\Delta)-f^*D$ is $\bQ$-effective and the stable base locus $\textbf{B}(-(K_X+\Delta)-f^*D)$ does not dominate $Y$, then $$\dim Y-\kappa(Y,-K_Y-D)\leq \dim X- \kappa (X,-(K_X+\Delta)-f^*D).$$
    
    \item Suppose $-(K_{X/Y}+\Delta)$ is $\bQ$-effective with ${\bf B}(-(K_{X/Y}+\Delta))$ that does not dominate $Y$, then for any effective $\bQ$-Cartier divisor $D_0$ on $Y$, we have $$\kappa(X,-(K_{X/Y}+\Delta)+f^*D_0)\leq \kappa(F,-(K_F+\Delta_F))+\kappa(Y,D_0).$$
\end{enumerate}
\end{corollary}

\begin{remark}
    There is another way to prove Theorem \ref{antiK}: First, we prove the theorem under the assumption that $f$ is equidimensional. Under this assumption, $f'$ is also equidimensional, hence so does $W'\rightarrow G'$. Therefore, we can prove the inequality holds for equidimensional morphism by the same steps of the above proof. Then for the general case, we can let $X'\rightarrow Y'$ be the normalization of the flattening $f$, and use the same argument to prove the inequality holds for $f$ by the fact the inequality holds for $f'$.
\end{remark}

\section{Further discussions and questions}

\subsection{Asymptotic invariants of the anti-canonical divisor}
In \cite[Corollary 3.5, Remark 3.6]{EIM}, Ejiri, Iwai, and Matsumura describe the asymptotic base locus of $-K_Y-D$ when $Y$ has at worst canonical singularities. In Theorem \ref{pe}, Theorem \ref{big}, and Theorem \ref{eff}, we proved the positivity of $-K_Y-D$ without assuming that $Y$ has at worst canonical singularities.  Thus, it is natural to ask the following question:

\begin{question} 
Can we describe the asymptotic base locus of $-K_Y-D$  without assuming that $Y$ has at worst canonical singularities?
\end{question}

It seems not easy to show that \cite[Corollary 3.5, Remark 3.6]{EIM} is still true when $Y$ has klt singularities or worse (if it is true). Naively thinking, if $Y$ has at worst canonical singularities, then for any resolution $\mu:Y'\rightarrow Y$, we have $\mu^*(-K_Y)\geq-K_{Y'}$. Therefore, if we can compute the intersection number and/or asymptotic invariants of $-K_{Y'}$ on some birational model $Y'$ (for example, as the proof of \cite[Theorem 4.1]{FG12}, which uses \cite[Theorem 2]{K98} to compute $(-K_Y.C)$ on a higher birational model), then it will give a lot of information about asymptotic invariants of $-K_Y$. But this approach does not work without assuming $Y$ has canonical singularities.

Also, comparing \cite[Corollary 4.7]{EG}, \cite[Theorem 1.2]{CCM} and \cite[Theorem 3.9]{EIM} to our results in Section 4, it seems that the behavior of nef (relative) anti-canonical divisors and effective (relative) anti-canonical divisor with the stable base locus that does not dominate $Y$ are very similar, so we can also ask the following question:

\begin{question} 
Does the inequality of Theorem \ref{antiK} still hold if $L$ is $\bQ$-effective and ${\bf B}_-(L)$ does not surject onto $Y$ (instead of the assumption ${\bf B}(L)$ does not dominate $Y$)?
\end{question}

\subsection{Rational map.}

Comparing \cite[Section 4 and Section 5]{EG} to Theorem \ref{antiK}, it seems natural to ask whether Theorem \ref{antiK} holds if $f$ is an almost holomorphic fibration. Unfortunately, we have the following counterexample:

\begin{example}
Let $X:=\bP^2\times \bP^1$, and $p_1:X\rightarrow \bP^2$ be the first projection. Then let $\mu:Y\rightarrow \bP^2$ be the blow-up of 13 general points $P_1,...,P_{13}$ on $\bP^2$, and denote $E_1,...,E_{13}$ be the corresponding exceptional divisors. Then we have $E_i^2=-1$ and $(E_i.E_j)=0$ if $i\neq j$. Since 14 general points define a quartic plane curve, for a general point $P$ on $\bP^2$, there is a quartic $C_0$ passing $P$ and all $P_i$ with multiplicity 1. Now, let $C$ be the proper transform of $C_0$ on $Y$. Then we have $(C.E_i)=1$ for all $i$, and
\begin{align*}
    (-K_Y.C)=(-\pi^*K_{\bP^2}-\sum E_i.\pi^*C_0-\sum E_i)=-1. 
\end{align*}
Therefore, $-K_Y$ is not almost nef, hence not pseudo-effective by \cite[Proposition 4.2 and 4.5]{BKKMSU}. Let $f:X\dashrightarrow Y$ be the induced rational map, which is almost holomorphic since the fibre is well-defined and closed over $Y-\Exc(\mu)$. Then $-K_X$ is ample, but $-K_Y$ is not effective, hence the inequality in Theorem \ref{antiK} does not hold for the map $f$.
\end{example}

This example also shows that the last assertion of Theorem \ref{pe}, Theorem \ref{big}, and Theorem \ref{eff} can not be generalized to the situation that $f$ is only an almost holomorphic fibration even if $D=0$. However, by Lemma \ref{eg1}, we have the following result, which is a generalization of \cite[Lemma 4.1]{Den}:

\begin{proposition}
Let $f:X\dashrightarrow Y$ be a birational map between normal projective varieties. Let $X_0,Y_0$ be the maximal open sets of $X,Y$ such that $f|_{X_0}:X_0\rightarrow Y_0$ is a morphism. Let $\Delta$ be a $\bQ$-divisor such that both $K_X+\Delta$ and $K_Y+f_*\Delta$ are $\bQ$-Cartier. Suppose that  $\pm(K_X+\Delta)$ is pseudo-effective (resp. effective, big), and $Y-Y_0$ has codimension at least 2, then $\pm(K_Y+f_*\Delta)$ is pseudo-effective (resp. effective, big). In particular, this result holds if $f$ is either a divisorial contraction or a $(K_X+\Delta)$-flip. 
\end{proposition}

\begin{proof}
We work on the case of anti-canonical divisors, and the case of canonical divisors can be derived in the same way. Let $g:W\rightarrow Y$ be a resolution of indeterminacy of $f$ such that $W$ is normal and denote $\pi:W\rightarrow X$ be the corresponding birational morphism.

Suppose that $-(K_X+\Delta)$ is pseudo-effective, then we write $$K_W+\Delta_W=\pi^*(K_X+\Delta)+E_X,$$ where $\Delta_W$ is the proper transform of $\Delta$ on $W$, and $E_X$ is $\pi$-exceptional. Then we have $-\pi^*(K_X+\Delta)=-K_W-\Delta_W+E_X$ is pseudo-effective. Next, let $(f_*\Delta)_W$ be the proper transform of $f_*\Delta$ on $W$. Then there is a $g$-exceptional divisor $E_Y$ such that $K_W+(f_*\Delta)_W=g^*(K_Y+f_*\Delta)+E_Y.$
Therefore,
\begin{align*}
    -g^*(K_Y+f_*\Delta)&=(-K_W-\Delta_W+E_X)+(\Delta_W-(f_*\Delta)_W)+E_Y-E_X.
\end{align*}
Note that $(\Delta_W-(f_*\Delta)_W)$ is $g$-exceptional, and since $Y-Y_0$ has codimension at least 2, every $\pi$-exceptional divisor is $g$-exceptional. So $-g^*(K_Y+f_*\Delta)=(-K_W-\Delta_W+E_X)$ + ($g$-exceptional divisors). Hence by Lemma \ref{eg1}(2), $-(K_Y+f_*\Delta)$ is pseudo-effective.

Suppose now that $-(K_X+\Delta)$ is effective. Given a birational morphism $\pi:X'\rightarrow X$, if $\pi^*D$ is effective outside the exceptional locus, then $D$ is effective itself. Thus, under the assumption that $-(K_X+\Delta)$ is effective, we can use the same argument of the pseudo-effective case to show that $-(K_Y+f_*\Delta)$ is effective, by using this fact at where we use Lemma \ref{eg1} in the pseudo-effective case.

Suppose now that $-(K_X+\Delta)$ is big. Since $Y-Y_0$ has codimension at least 2, for any $\bQ$-divisor $G$ on $Y$, there is a $\bQ$-divisor $G'$ on $X$ such that $f_*G'=G$. Since $-(K_X+\Delta)$ is big, we have $-(K_X+\Delta+\varepsilon G')$ is pseudo-effective for $\varepsilon$ sufficiently small. Thus, replacing $\Delta$ by $\Delta+\varepsilon G'$, the pseudo-effective case implies $-(K_Y+f_*(\Delta+\varepsilon G'))=-(K_Y+f_*\Delta)-\varepsilon G$ is pseudo-effective. Since $G$ is arbitrary, this implies $-(K_Y+f_*\Delta)$ is big. 
\end{proof}

By the above proposition, it is natural to ask whether we can generalize Theorem \ref{pe}, Theorem \ref{big}, and Theorem \ref{eff} to almost holomorphic fibration with $Y-Y_0$ has codimension at least 2. To answer this question, first, we define the pullback of divisors under rational maps: 

\begin{definition}
(cf. \cite[Definition 1.2]{M}) Let $f:X\dashrightarrow Y$ be a rational map. Then for a ($\bQ$-)divisor $D$ on $Y$, we can define the pullback $f^*D$ in the following way: Let $\pi:X'\rightarrow X$ be a resolution of indeterminacy on $f$, and $f':X'\rightarrow Y$ be the corresponding resolution. Then we define $f^*D:=\pi_*f'^*D$. Note that this definition does not depend on the choice of resolution because the push-forward of exceptional divisors is zero.
\end{definition}

Also, we need the following lemma:

\begin{lemma}\label{effexc}
Let $f:X\rightarrow Y$ be an algebraic fibre space between normal varieties. Let $E$ be an effective $f$-exceptional $\bQ$-divisor on $X$, and $D$ be a $\bQ$-divisor on $Y$. Suppose $f^*D+nE$ is $\bQ$-effective for some $n\in \bN$, then $D$ is $\bQ$-effective.
\end{lemma}
\begin{proof}
Let $m\in \bN$ sufficiently divisible such that $mD$ and $mE$ has integer coefficients, and $m(f^*D+nE)$ is effective. Since $f$ is an algebraic fibre space and $f(\Supp E)$ has codimension at least 2 in $Y$, letting $Y_0:=Y-f(\Supp E)$, $X_0:=X-f^{-1}f(\Supp E)$, we have the following natural isomorphisms: $$H^0(X_0,f^*(mD)|_{X_0})\cong H^0(Y_0,(mD)|_{Y_0})\cong H^0(Y,mD) \cong H^0(X,f^*(mD)) .$$
In particular, if $(f^*(mD))|_{X_0}$ is effective, then so does $mD$. Thus, if $m(f^*D+nE)$ is effective, then $(f^*(mD))|_{X_0}$, and hence $mD$ itself is effective.
\end{proof}

Now we have the following generalization:

\begin{proposition}
Let $f:X\dashrightarrow Y$ be an almost holomorphic fibration between normal projective varieties, with $X_0,Y_0$ be the maximal open subsets of $X$ and $Y$ such that $f|_{X_0}:X_0\rightarrow Y_0$ is a morphism. Suppose that $Y-Y_0$ has codimension at least 2, then the conclusions of Theorem \ref{pe}, \ref{big}, and \ref{eff} still hold by replacing the assumption on ${\bf B}_-(L)$ (resp. ${\bf B}_+(L), {\bf B}(L)$) with the one that ${\bf B}_-(L)\cap X_0$ (resp. ${\bf B}_+(L)\cap X_0,$ ${\bf B}(L)\cap X_0$) does not surject onto $Y_0$. 
\end{proposition}
\begin{proof}
Let $\pi:X'\rightarrow X$ be a resolution of indeterminacy, $f':X'\rightarrow Y$ be the corresponding resolution. We can write $-(K_{X'}+\Delta')-f'^*D=\pi^*L+E^+-E^-$ for some effective $\pi$-exceptional divisor $E^+,E^-$, where $\Delta'$ is the proper transform of $\Delta$ on $X'$. Since $f$ is almost holomorphic, we may assume $f'(\Exc(\pi))\subset Y-Y_0$. In particular, every $\pi$-exceptional divisors are $f'$-exceptional.  

Now, suppose $L$ is pseudo-effective and $({\bf B}_-(L)\cap X_0)$ does not surject onto $Y_0$. Then ${\bf B}_-(-(K_{X'}+\Delta'+E^+-E^-)-f'^*D)={\bf B}_-(\pi^*L)$ is not surject onto $Y$ since ${\bf B}_-(\pi^*L)\subset \pi^{-1}({\bf B}_-(L))$. In fact, let $H$ be an ample Cartier divisor on $X$ and $A$ be an ample Cartier divisor on $X'$ such that $A-\pi^*H$ is also ample. Then for any $0 < \varepsilon \ll 1,$ we have 
\begin{align*}
{\bf B}(\pi^*L+\varepsilon A)&={\bf B}(\pi^*(L+\varepsilon H)+\varepsilon(A-\pi^*H))\\
&\subset {\bf B}(\pi^*(L+\varepsilon H))\subset \pi^{-1}({\bf B}(L+ \varepsilon H))\subset \pi^{-1}({\bf B}_-(L)).
\end{align*}
Hence ${\bf B}_-(\pi^*L)\subset \pi^{-1}({\bf B}_-(L))$ follows from the definition of ${\bf B}_-(-)$. Since $f(\Supp E)\neq Y$, we can apply Theorem \ref{pe} to conclude that for sufficiently divisible positive integer $l$, $\cO_{X'}(l(f'^*(-K_Y-D)+\Delta'^-+E^-+B))$ is weakly positive for some effective $f'$-exceptional divisor $B$, hence so is $\cO_X(l(f^*(-K_Y-D)+\Delta^-+\pi_*B))$. In particular, since $E^-$ is also $f'$-exceptional, by the weakly positivity of $\cO_{X'}(l(f'^*(-K_Y-D)+\Delta'^-+E^-+B))$ and Lemma \ref{eg1}(2), we conclude that if $\Delta$ is effective, then $-K_Y-D$ is pseudo-effective, which generalizes Theorem \ref{pe} (and hence Theorem \ref{big}). Using a similar method, and applying Lemma \ref{effexc}, we can also generalize Theorem \ref{eff}.
\end{proof}

\begin{remark}
{\em The above proof shows that even if we only assume $f$ is almost holomorphic,  $\cO_X(l(f^*(-K_Y-D)+\Delta^-+\pi_*B))$ is still weakly positive. But without the assumption of $Y-Y_0$ having codimension at least 2, then the pseudo-effectiveness (resp. effectiveness, bigness) of $\cO_X(l(f^*(-K_Y-D)+\pi_*B))$ and $\cO_{X'}(l(f'^*(-K_Y-D)+E^-+B))$ is not enough to imply the pseudo-effectiveness (resp. effectiveness, bigness) of $-K_Y-D$ even if $\Delta$ is effective. However, if $f$ is only a rational map, then $\Supp(E^-)$ may map onto $Y$ and hence the weak positivity of $\cO_X(l(f^*(-K_Y-D)+\Delta^-+\pi_*B))$ may not be true. }
\end{remark}

\begin{proposition}\label{alho}
The inequality of Theorem \ref{antiK} holds if $f$ is an almost holomorphic fibration with $Y-Y_0$ has codimension at least 2.
\end{proposition}
\begin{proof}
The proof is similar to the argument generalizing Theorem \ref{antiK} from smooth cases to the general case. Consider the following diagram
\begin{center}
    \begin{tikzcd}
&X'\arrow[r,"\pi"]\arrow[d,"f'"] &X\arrow[d,"f", dashed]\\
&Y'\arrow[r,"\mu"] &Y.
\end{tikzcd}
\end{center}

Where $Y'$ is log resolution of $(Y,Y-Y_0)$, and $X'$ is a normalization of the resolution of indeterminacy of $X \rightarrow Y'$. We may assume for the induced morphisms $\pi:X'\rightarrow X$ and $f': X'\rightarrow Y'$, $\Exc(\pi)$ is purely codimension 1 and $f'(\Exc (\pi))\subset \Exc(\mu)$. Now we let $L=-(K_X+\Delta)-f^*D$, and write $\pi^*L=-(K_{X'}+\Delta')+E^+-E^--f'^*\mu^*D$, where $\Delta'$ is the proper transform of $\Delta$ on $X'$, and $E^+$, $E^-$ are effective $\pi$-exceptional $\bQ$-divisors defined by the equality $K_{X'}+\Delta'=\pi^*(K_X+\Delta)+E^+-E^-$. Note that by the construction of $X'$, both $E^+$ and $E^-$ are vertical with $f(\Supp E^+), f(\Supp E^-)\subset \Exc(\mu)$, hence there exists an effective $\mu$-exceptional divisor $N$ on $Y'$ such that $f'^*N\geq E^+\geq E^+-E^-$. Now, since the theorem holds if $f$ is an algebraic fibre space, we have 
\begin{align*}
\kappa(X,L)=\kappa(X',\pi^*L)&\leq \kappa(X',-(K_{X'}+\Delta')-f^*\mu^*D+f'^*N)\\
&\leq \kappa(F,-(K_F+\Delta_F))+\kappa(Y',-K_{Y'}-\mu^*D+N).
\end{align*}
Note that $-(K_{X'}+\Delta')-f^*\mu^*D+f'^*N=\pi^*L+(f^*N-(E^+-E^-))$ is still effective with the stable base locus does not dominant $Y'$. Now, we write $K_{Y'}=\mu^*K_Y+B^+-B^-$ for some $\mu$-exceptional divisors $B^+,B^-$. Then we have 
\begin{align*}
\kappa(Y',-K_{Y'}-\mu^*D+N)&=\kappa(Y',\mu^*(-K_Y-D)+N+B^--B^+)\\
 &\leq \kappa(Y',\mu^*(-K_Y-D)+N+B^-)\\
 &=\kappa(Y',\mu^*(-K_Y-D))\\
 &= \kappa(Y,-K_Y-D)
\end{align*}
 since $N$ and $B$ are $\mu$-exceptional, hence the proof completes.
\end{proof}

\end{document}